\documentclass[letter,12pt]{article}

\usepackage[english]{babel}
\usepackage{amssymb}
\usepackage{amsthm}
\usepackage{amsfonts}
\usepackage{amsmath}
\usepackage{anysize}
\usepackage{bbm}
\usepackage{abstract}
\marginsize{2.5cm}{2.5cm}{0.9cm}{1.9cm}
\usepackage{upquote}
\usepackage{xcolor}

\newtheorem{mtheorem}{Theorem}

\newtheorem{theorem}{Theorem}[section]
\newtheorem{lemma}[theorem]{Lemma}
\newtheorem{prop}[theorem]{Proposition}
\newtheorem{coro}[theorem]{Corollary}

\theoremstyle{definition}

\newtheorem{remark}[theorem]{Remark}

\newcommand{\vep}{\varepsilon}
\newcommand{\psh}{{\rm PSH}}
\newcommand{\PSH}{{\rm PSH}}

\newcommand{\ddbar}{\partial\bar\partial}

\newcommand{\Amp}{{\rm Amp}}

\newcommand{\R} {{\bf  R}}

\usepackage{hyperref}
\hypersetup{
    bookmarks=true,         
    unicode=false,          
    pdftoolbar=true,        
    pdfmenubar=true,        
    pdffitwindow=false,     
    pdfstartview={FitH},    
    pdftitle={The metric geometry of singularity types},    
    pdfauthor={T. Darvas, E. Di Nezza, C.H. Lu},     
    colorlinks=true,       
   linkcolor=black,          
    citecolor=black,        
    filecolor=black,      
    urlcolor=black}           

\title {Monge-Amp\`{e}re measures on contact sets}
\author{ Eleonora Di Nezza \& Stefano Trapani }
\date{\vspace{-0.5cm}}

\begin{document}

\maketitle

\begin{abstract}
\noindent Let $(X, \omega)$ be a compact K\"ahler manifold of complex dimension $n$ and $\theta$ be a smooth closed real $(1,1)$-form on $X$ such that its cohomology class $\{ \theta \}\in H^{1,1}(X, \R)$ is  pseudoeffective. 
Let $\varphi$ be a $\theta$-psh function, and let $f$ be a continuous function on $X$ with bounded distributional laplacian with respect to $\omega$ such that $\varphi \leq f. $
 Then the non-pluripolar  measure  
$\theta_\varphi^n:= (\theta + dd^c \varphi)^n$ satisfies the equality: 
$$ {\bf{1}}_{\{ \varphi = f \}} \ \theta_\varphi^n =   {\bf{1}}_{\{ \varphi = f \}} \ \theta_f^n,$$ 
where, for a subset $T\subseteq X$, ${\bf{1}}_T$ is the characteristic function. 
In particular we prove that 
\[ \theta_{P_{\theta}(f)}^n= { \bf {1}}_{\{P_{\theta}(f) = f\}} \   \theta_f^n\qquad {\rm and }\qquad \theta_{P_\theta[\varphi](f)}^n = { \bf {1}}_{\{P_\theta[\varphi](f) = f \}} \  \theta_f^n. \] 
\end{abstract}

\section{Introduction}
Starting from the works of Zaharjuta\cite{Z76} and Siciak \cite{S77}, that years later have been take over by Bedford and Taylor \cite{BT,BT82,BT86}, \emph{envelopes of plurisubharmonic functions} started to be of interest and to play an important role in the development of the pluripotential theory on domains of $\mathbb{C}^n$. 

When, relying on the Bedford and Taylor theory in the local case, the foundations of a pluripotential theory on compact  K\"ahler manifolds has been developed \cite{GZ05, GZ07}, \emph{envelopes of quasi-plurisubharmonic functions} started to be intensively studied.

As geometric motivations we can mention, among others, the study of geodesics in the space of K\"ahler metrics \cite{Chen00, Dar17, Ber17, RWN, CTW18, DDL1, CMc19} and the trascendental holomorphic Morse inequalities on projective manifolds \cite{WN}.

The two basic (and related) questions are about the regularity of envelopes and the behaviour of their Monge-Amp\`ere measures. 
To fix notations, let $(X,\omega)$ be a compact K\"ahler manifold of complex dimension $n$, $\theta$ be a smooth closed real $(1,1)$-form and let $f$ be a function on $X$ bounded from above. We are going to refer to $f$ as ``barrier function". Then the ``prototype" of an envelope construction is 
$$ P_\theta(f):= \left(\{ u\in \PSH(X, \theta), \; u\leq f \}\right)^*.$$ 
Such a function is either a genuine $\theta$-plurisubharmonic function or identically $-\infty$. 
When $ f= -{\bf 1}_T$ is the negative characteristic function of a subset $T$, then $P_\theta (f) = f^*_T$ is the so called relative extremal function of $T$ \cite{GZ05}.
When $f = 0$ then $P_\theta(0)= V_\theta$ is a distiguished potential with minimal singularities. 
\\The study of such envelopes has lead to several works. We start summarizing them in the case of a smooth barrier function $f$.\\
The first result to mention is \cite{Ber09} where the author proves that in the case $\theta\in c_1(L)$ where $L$ is a big line bundle over $X$, the envelope $P_\theta(f)$ is  ${C}^{1,1}$ on $ \Amp(\{\theta\})$ and moreover 

\begin{equation}\label{support}
\theta_{P_\theta(f)}^n= {\bf 1}_{\{P_\theta(f)=f\}} \theta_f^n.
\end{equation}
After \cite{BD12}, people started to work on possible generalisations of the above results in the case of a pseudoeffective class $\{\theta\}$, that does not necesseraly represents the first Chern class of a line bundle. If we assume $\{\theta\}$ big and nef, Berman \cite{Ber13}, using PDE methods, proved that the envelope $P_\theta(f)$ is in $C^{1,\alpha}$ on $\Amp(\{\theta\})$ for any $\alpha\in (0,1)$ and the identity in \eqref{support} holds. The optimal regularity $C^{1,1}$ in the K\"ahler case was then proved indipendently by \cite{T18} and \cite{CZ19} while the big and nef case was settled in \cite{CTW18}.\\
For general pseudoeffective classes only the following inequality of measures is known \cite{DDL1}:
\begin{equation}\label{support1}
\theta_{P_\theta(f)}^n\leq {\bf 1}_{\{P_\theta(f)=f\}} \theta_f^n.
\end{equation}

In the study of geodesics, \cite{RWN} introduced another type of envelope: given a $\theta$-plurisubharmonic function $\varphi$, the so called \emph{maximal envelope} is defined as
$$P_\theta[\varphi](f):=\left(\lim_{C\rightarrow+\infty} P_\theta(\min(\varphi+C, f))\right)^*.$$ 
In the same paper, under the assumption $\{\theta\}=c_1(L)$, they proved the equality
\begin{equation}\label{support2}
\theta_{P_\theta[\varphi](f)}^n= {\bf 1}_{\{P_\theta[\varphi](f)=f\}} \theta_f^n.
\end{equation}
In the case of general pseudoeffective classes, the inequality $\leq $ in \eqref{support2} was proved in \cite{DDL2}, whereas the equality is derived in \cite{Mc} when $\varphi$ has analytic singularities.

In the literature, regularity questions about envelopes for functions $f$ that are less regular have been also adressed: in the case $\theta$ is a K\"ahler form, Darvas and Rubinstein \cite{DR} proved that if $f$ is $C^1$/$C^{1, \overline{1}}$, its envelope is $C^1$/$C^{1, \overline{1}}$ as well; while Guedj, Lu and Zeriahi \cite{GLZ17} proved that if $f$ is a continuous function, $P_\theta(f)$ is also continuous. They also proved that its Monge-Amp\`ere measure (w.r.t. any big class $\{\theta\}$) is supported on the contact set $\{ P_\theta (f)=f\} $. 

In the present paper we prove the following theorem:

\begin{mtheorem}
Let $ \theta$ be  smooth closed real $(1,1)$-form on $X$ such that  the cohomology class $ \{ \theta \}$ is  pseudoeffective.  Let $\varphi$ be a $\theta$-plurisubharmonic function and $f\in C^{1, \overline{1}}$(X). Then the non-pluripolar product  $\theta_{ \varphi}^n$ satisfies the equality

$${\bf 1}_{\{ \varphi = f \} } \theta^n_{ \varphi} ={\bf 1}_{\{ \varphi= f \} } \theta^n_{ f}.
$$
\end{mtheorem}

\noindent In particular, when the barrier function $f$ is in $ C^{1, \overline{1}}(X)$ we obtain the equality in \eqref{support} and \eqref{support2} in the general pseudoeffective case. Note that, at the best of our knowledge, the equality in \eqref{support2} is new even in the case of a K\"ahler class.

\medskip\noindent\textbf{Acknowledgements.} 
We thank Chinh Lu for his comments on a previous draft of this paper and in particular for sharing Remark \ref{example}.

\section{  Preliminaries}

Let $(X,\omega)$ be a compact K\"ahler manifold of dimension $n$ and fix $\theta$ a smooth closed real $(1,1)$-form.
A function $\varphi: X \rightarrow \mathbb{R}\cup \{-\infty\}$ is called \emph{quasi-plurisubharmonic} (qpsh for short) if locally $\varphi= \rho +u$, where $\rho$ is smooth and $u$ is a plurisubharmonic function. We say that $\varphi$ is $\theta$-plurisubharmonic ($\theta$-psh for short) if it is quasi-plurisubharmonic and $\theta_\varphi:=\theta+i\ddbar \varphi  \geq 0$ in the weak sense of currents on $X$. We let $\PSH(X,\theta)$ denote the space of all $\theta$-psh functions on $X$. The class $\{\theta\}$ is {\it pseudoeffective} if $\PSH(X,\theta)\neq \emptyset$ and it is {\it big} if there exists $\psi\in \psh(X,\theta)$ such that $\theta +i\ddbar \psi\geq \vep \omega$ for some $\vep>0$, or equivalently if $\PSH(X,\theta-\varepsilon\omega)\neq \emptyset$.   
 

When $\theta$ is non-K\"ahler, elements of $\textup{PSH}(X,\theta)$ can be quite singular, and we distinguish the potential with the smallest singularity type in the following manner:
$$V_\theta := \sup \{u \in \textup{PSH}(X,\theta) \textup{ such that } u \leq 0\}.$$
A function $\varphi \in \PSH(X,\theta)$ is said to have minimal singularities if it has the same singularity type as $V_{\theta}$, i.e., $|\varphi-V_\theta|\leq C$ for some $C>0$. Note that, given any $\theta$-psh function $\varphi $ we have $\varphi -\sup_X \varphi \leq V_\theta$.

Given $\theta^1,...,\theta^n$ closed smooth real $(1,1)$-forms representing pseudoeffective cohomology classes and $\varphi_j \in \textup{PSH}(X,\theta^j)$, $j=1,...n$, following the construction of Bedford-Taylor \cite{BT,BT82} in the local setting, it has been shown in \cite{BEGZ10} that the sequence of positive measures
\begin{equation}\label{eq: k_approx_measure}
{\bf 1}_{\bigcap_j\{\varphi_j>V_{\theta^j}-k\}}\theta^{1}_{\max(\varphi_1, V_{\theta^1}-k)}\wedge \ldots\wedge \theta^n_{\max(\varphi_n, V_{\theta^n}-k)}
\end{equation}
has total mass (uniformly) bounded from above and is non-decreasing in $k \in \Bbb R$, hence converges weakly to the so called \emph{non-pluripolar product} 
\[
\theta^1_{\varphi_1 } \wedge\ldots\wedge\theta^n_{\varphi_n }.
\]
The resulting positive measure does not charge pluripolar sets. In the particular case when $\varphi_1=\varphi_2=\ldots=\varphi_n=\varphi$ and $\theta^1=...=\theta^n=\theta$ we will denote by $ \theta_{\varphi}^n$ the non-pluripolar Monge-Amp\`ere measure of $\varphi$. As a consequence of Bedford-Taylor theory it can be seen that the measures in \eqref{eq: k_approx_measure} all have total mass less than $\int_X \theta^1_{V_{\theta^1}}\wedge \ldots \wedge \theta^n_{V_{\theta^n}}$, in particular, after letting $k \to \infty$ 
$$\int_X \theta^1_{\varphi_1}\wedge \ldots \wedge \theta^n_{\varphi_n}\leq \int_X \theta^1_{V_{\theta^1}}\wedge \ldots \wedge \theta^n_{V_{\theta^n}}.$$
We recall that the \emph{plurifine topology} is the weakest topology for which qpsh functions are continuous and that the non-pluripolar Monge-Amp\`ere measure satisfies a locality condition with respect to the plurifine topology \cite[Section 1.2]{BEGZ10}, i.e. if $\varphi_j,\psi_j,j=1,...,n$, are $\theta^j$-psh functions such that $\varphi_j =\psi_j$ on $U$ an open set in the plurifine topology, then 
\begin{equation}\label{plurifine}
{\bf{1}}_{U} \theta^1_{\varphi_1} \wedge ... \wedge \theta^n_{\varphi_n} ={\bf{1}}_{U} \theta^1_{\psi_1} \wedge ... \wedge \theta^n_{\psi_n}.
\end{equation}
We also note that sets of the form $\{\varphi<\psi\}$, where $\varphi, \psi$ are qpsh functions, are open in the plurifine topology. 

In the following we are going to work with some well known envelope constructions:
$$P_\theta(f) , \ P_\theta(f_1,\ldots, f_k), \ P_\theta[\varphi](f),  \ P_\theta[\varphi].$$ 
Given $f,f_1,\ldots f_k $ functions on $X$ bounded from above, we consider the ``rooftop envelopes'' 
$$ P_\theta(f):=(\sup\{v \in \textup{PSH}(X,\theta), \ v \leq f) \})^*$$
and 
$$P_\theta(f_1, \ldots, f_k ):=P_\theta(\min(f_1,\ldots , f_k))=(\sup\{v \in \textup{PSH}(X,\theta), \ v \leq \min(f_1, \ldots f_k) \})^*.$$
Then, given a $\theta$-psh function $\varphi$, the above procedure allows us to introduce
$$P_\theta[\varphi](f) := \Big(\lim_{C \to +\infty}P_\theta(\varphi+C,f)\Big)^*.$$
Note that by definition we have $ P_\theta[\varphi](f) = P_\theta[\varphi](P_{\theta}(f)). $ 
When $f =0$, we will simply write $P_\theta[\varphi]:=P_\theta[\varphi](0)$.
We emphasize that the functions $P_\theta(f)$, $P_\theta(f_1, \ldots, f_k )$ and $P_\theta[\varphi](f)$ are either $\theta$-psh or identically equal to $-\infty$. Moreover, observe that if $-C\leq f \leq C$, then $V_\theta-C \leq P_\theta(f) \leq V_\theta+C$; hence $P_\theta(f)$ is a well defined $\theta$-psh function. Morever, if we assume $\varphi\leq f$ the following sequence of inequalities holds:
\begin{equation}\label{disegEnv}
\varphi \leq P_\theta[\varphi](f)\leq P_\theta(f)\leq f.
\end{equation}
As in \cite[Section 2]{DR} we denote by $C^{1, \overline{1}}(X)$ the space of continuous function with bounded distributional laplacian w.r.t. $\omega$. Elliptic regularity and  Sobolev's embedding theorem imply that $C^{1, \overline{1}}(X)\subset W^{2,p}\subset C^{1, \alpha}$ for any $p\geq 1$ and $\alpha\in (0,1)$. Here $W^{2,p}$ denotes the Sobolev space of functions with all derivatives up to second order in $L^p$. By H\"older inequality, any polynomial having as coefficients the second derivatives of $f$ is in any $L^q$, $q\geq 1$. In particular $(\theta +dd^c f)^n=h \,\omega^n$ where $h\in L^q$.

%
\section{Monge-Amp\`ere measures}
The starting point is given by the following Lemma that deeply relies on \cite[Theorem 3.4]{ Ber09} and on \cite[Theorem 2.5]{DR}:

\begin{lemma}\label{min}
Let $f_1, f_2\in C^{1,\bar{1}}(X)$. Then 
$P_{\omega}(f_1,f_2)$ is also $C^{1,\bar{1}},$ and for $i= 1,2$  the functions $f_i$ and   $P_{\omega}(f_1,f_2)$ are equal up to second order at almost every point  on the set  $\{ P_{\omega}(f_1,f_2)=f_i  \}$. In particular, the functions $f_1, f_2, P_{\omega}(f_1,f_2)$ are equal up to second order at almost every point on the set $\{ P_{\omega}(f_1,f_2)=f_1 = f_2  \}$. In particular 
 \[ {\bf 1}_{ \{ P_{\omega}(f_1,f_2)  = f_1 = f_2 \} } \omega_{f_1}^n  = {\bf 1}_{ \{ P_{\omega}(f_1,f_2)  = f_1 = f_2 \} } \omega_{f_2}^n.  \] 
Moreover the measures 
$$ {\bf 1}_{ \{ P_{\omega}(f_1,f_2)  = f_1 \} }  \omega_{f_1}^n, \quad  {\bf 1}_{ \{ P_{\omega}(f_1,f_2)  = f_2 \} }  \omega_{f_2}^n, \quad  {\bf 1}_{ \{ P_{\omega}(f_1,f_2)  = f_1 = f_2 \} }  \omega_{f_j}^n\quad (j=1,2)   $$
are positive and 
\begin{eqnarray*}
{\omega}_{P_{\omega}(f_1,f_2)}^n &=& {\bf 1}_{ \{ P_{\omega}(f_1,f_2)  = f_1 \} }  \omega_{f_1}^n+  {\bf 1}_{ \{ P_{\omega}(f_1,f_2)  = f_2 \} }  \omega_{f_2}^n -  {\bf 1}_{ \{ P_{\omega}(f_1,f_2)  = f_1 = f_2 \} } \omega_{f_1}^n  \\ 
&=& {\bf 1}_{ \{ P_{\omega}(f_1,f_2)  = f_1 \} }  \omega_{f_1}^n +  {\bf 1}_{ \{ P_{\omega}(f_1,f_2)  = f_2 \} }  \omega_{f_2}^n -  {\bf 1}_{ \{ P_{\omega}(f_1,f_2)  = f_1 = f_2 \} }  \omega_{f_2}^n
\end{eqnarray*}

\end{lemma}

\begin{proof}

Observe that  $P_{\omega}(f_1,f_2)$ is a genuine $\omega$-psh function. Moreover, by \cite[Theorem 2.5]{DR} the function $P_{\omega}(f_1,f_2)\in C^{1,\overline{1}}(X)$, hence the measures $\omega_{f_1}^n$, $ \omega_{f_2}^n$ and $\omega_{P_{\omega}(f_1,f_2)}^n$ are absolutely continuous with respect to the Lebesgue measure. We set 
$\Psi_1  = P_{\omega}(f_1,f_2) - f_1, \Psi_2 = P_{\omega}(f_1,f_2) - f_2 $ and $\Psi_3 = f_1 - f_2.$  Since the functions $\Psi_i$ are $C^1$ the sets $A_i =  \{ \Psi_i = 0,  d \Psi_i \neq 0 \}$ are real hypersurfaces and they have Lebesgue measure zero. Since $\Psi_i \in W^{2,p}$, the set $B_i$ where the real hessian matrix of  $\Psi_i$ does not exists also have Lebesgue measure zero. Finally by \cite[page 53, Lemma A4]{KS}, the set $C_i$  where $\Psi_i = 0, d \Psi_i = 0 $ and  the real hessian matrix of $\Psi_i$  exists but it is non zero, has Lebesgue measure zero as well.  
Then for $1 \leq i \leq 3,$ the function $\Psi_i$ is zero up to order two almost everywhere on the set 
$\{ \Psi_i = 0 \}.$ 
 Furthermore by \cite[Corollary 9.2]{BT82}  the measure 
$ \omega_{P_{\omega}(f_1,f_2)}^n$ is supported on the set $\{ \Psi_1 = 0 \} \cup \{ \Psi_2 = 0 \}.$ Set $U_1 = \{ f_1 < f_2 \}, U_2 = \{ f_1 > f_2 \},$ and note that$U_1$ and $U_2$ are open sets such that $U_1 \cap \{ \Psi_2 = 0 \} = \emptyset$ and $ U_2 \cap \{ \Psi_1 = 0 \} = \emptyset$. Then we can argue that 
$$ 
{\bf 1}_{ U_1 \cap  \{ \Psi_1 = 0\} } \omega_{P_{\omega}(f_1,f_2)}^n = {\bf 1}_{U_1 \cap  \{ \Psi_1 = 0 \} } \omega_{f_1}^n, \quad  {\bf 1}_{U_2 \cap \{ \Psi_2 = 0 \} }\omega_{P_{\omega}(f_1,f_2)}^n = {\bf 1}_{U_2 \cap \{ \Psi_2 = 0 \} } \omega_{f_2}^n.
$$
This implies that the measures $ {\bf 1}_{U_1 \cap  \{ \Psi_1 = 0 \} } \omega_{f_1}^n$ and $   {\bf 1}_{U_2 \cap \{ \Psi_2 = 0 \} }\omega_{f_2}^n  $ are positive.\\
By the above argument, we can also garantee that at almost every point of the set  
$(\{ \Psi_1 = 0 \} \cup \{ \Psi_2 = 0 \}) \cap \{ \Psi_3 = 0 \} = \{ P_{\omega}(f_1,f_2) = f_1 = f_2 \} $ the functions $f_1, f_2$ and $P_{\omega}(f_1,f_2)$ coincide up to order two. Therefore we also have 
$$ {\bf 1}_{ \{ P_{\omega}(f_1,f_2) = f_1 = f_2 \}} \omega_{P_{\omega}(f_1,f_2)}^n = {\bf 1}_{ \{ P_{\omega}(f_1,f_2) = f_1 = f_2 \}} \omega_{f_1}^n  =  {\bf 1}_{ \{ P_{\omega}(f_1,f_2) = f_1 = f_2 \}} \omega_{f_2}^n, $$  
 and in particular it follows that, for any $j=1,2$, the measure $ {\bf 1}_{ \{ P_{\omega}(f_1,f_2)  = f_1 = f_2 \} } \omega_{f_j}^n$ is positive.
 Combining all the above equalities we get that for any $j=1,2$,
 \begin{align*}
 &\omega_{P_{\omega}(f_1,f_2)}^n \\
 &= {\bf 1}_{U_1 \cap  \{ \Psi_1 = 0 \} } \omega_{f_1}^n+  {\bf 1}_{U_2 \cap \{ \Psi_2 = 0 \} }\omega_{f_2}^n +{\bf 1}_{ \{ P_{\omega}(f_1,f_2)  = f_1 = f_2 \} } \omega_{f_j}^n\\
 &={ \bf 1}_{  \{ \Psi_1 = 0 \} } \omega_{f_1}^n \!-\! {\bf 1}_{ \{ P_{\omega}(f_1,f_2)  = f_1 = f_2 \} } \omega_{f_1}^n\! + \! {\bf 1}_{\{ \Psi_2 = 0 \} }\omega_{f_2}^n\! -\! {\bf 1}_{ \{ P_{\omega}(f_1,f_2)  = f_1 = f_2 \} } \omega_{f_2}^n\!+\!{\bf 1}_{ \{ P_{\omega}(f_1,f_2)  = f_1 = f_2 \} } \omega_{f_j}^n.
  \end{align*}
\end{proof}

\begin{theorem}\label{main} 
Let $ \theta^1,  \ldots, \theta^n $ be  smooth closed real $(1,1)$-forms  on $X$ such that  the cohomology classes  
$ \{ \theta^1 \}, \ldots, \{ \theta^n\} $ are  pseudoeffectives.  For $1 \leq i \leq n$, let $\varphi_i $ be a $\theta^i$-psh function and $f_i$ be a $C^{1,\bar{1}}$ function on $X$ such that $\varphi_i \leq f_i$.   Then the non-pluripolar product  
$\theta^1_{ \varphi_1} \wedge  \ldots \wedge \theta^n_{ \varphi_n}$ satisfies the equality: 

\begin{equation}\label{eqma} {\bf 1}_{\bigcap_j\{ \varphi_j = f_j \} } \theta^1_{ \varphi_1} \wedge \ldots \wedge \theta^n_{ \varphi_n} ={\bf 1}_{\bigcap_j\{ \varphi_j = f_j \} } \theta^1_{ f_1}  \wedge \ldots \wedge \theta^n_{ f_n} \end{equation}
\end{theorem} 
\noindent The proof below is inspired by \cite{WN}.

\begin{proof}
{\bf Step 1}. We start proving the case when $\theta:=\theta^1=\theta^2=\ldots =\theta^n$ is a K\"ahler form, $\varphi:=\varphi_1=\varphi_2=\ldots =\varphi_n$ and $f:=f_1=f_2=\ldots =f_n$ are both $\theta$-psh functions. Let $\phi_j$ be a sequence of smooth functions on $X$ decreasing to $\varphi$ and define $\psi_j:= P_{\theta}(\phi_j,f)$. Note that, since $\varphi$ is $\theta$ -psh and $\varphi\leq \phi_j, f$, we have
\begin{equation} \label{mainineq}
 \varphi  \leq \psi_j \leq  \min(\phi_j,f) \leq f.
\end{equation}
In particular,
\begin{equation}\label{envelop}
\{ \varphi = f \} \subseteq \{\psi_j = \min(\phi_j,f)  \} \cap \{  \min(\phi_j,f) = f \}  =  \{ \psi_j = f \} \cap \{ \phi_j \geq  f \} .
 \end{equation}  
Moreover, thanks to \eqref{mainineq} we can infer that the functions $\psi_j$ are decreasing to $\varphi$ as $j$ goes to $+ \infty$.
From Lemma \ref{min} we then get 
\begin{equation}\label{ineq1}
\theta_{\psi_j}^n =  {\bf 1}_{ \{ \psi_j  = f \} } \theta_f^n +  {\bf 1}_{ \{ \psi_j  = \phi_j \} } \theta_{\phi_j}^n -   {\bf 1}_{ \{ \psi_j  = \phi_j=f \} } \theta_{\phi_j}^n  \geq  {\bf 1}_{ \{ \psi_j  = f \} } \theta_f^n \geq   {\bf 1}_{ \{ \varphi = f \} }\theta_f^n, 
\end{equation}
where the last inequality follows from \eqref{envelop}.\\
Fix $C>0 $ such that $\min f  >  -C$ and $g$  a non negative continuous function on $X$. By \cite[Theorem 4.26]{GZ17} we know that for any sequence of uniformly bounded quasi-continuous functions $\chi_j$ converging in capacity to a bounded quasi-continuous function $\chi$, the measure
 $\chi_j\theta_{\max( \psi_j,-C)}^n$ weakly converges  to $\chi\theta_{\max( \varphi,-C)}^n$ as $j $ goes to $+ \infty$.
  Fix $\varepsilon>0$. We set
  $$h_j^{C, \varepsilon}:= \frac{\max(\psi_j+C,0)}{\max(\psi_j+C,0)+\varepsilon}, \quad h^{C, \varepsilon}:= \frac{\max(\varphi+C,0)}{\max(\varphi+C,0)+\varepsilon}$$ and we observe that $h_j^{C, \varepsilon}, h^{C, \varepsilon}$ are quasi-continuous (uniformly) bounded functions (with values in $[0,1]$) and that $h_j^{C, \varepsilon}$ converges in capacities to  $h^{C, \varepsilon}$. The last statement follows from the fact that 
  $$\tilde{h}_j^{C, \varepsilon}:= \frac{\max(\varphi+C,0)}{\max(\psi_j+C,0)+\varepsilon} \leq h_j^{C, \varepsilon}\leq \frac{\max(\psi_j+C,0)}{\max(\varphi+C,0)+\varepsilon}:= \hat{h}_j^{C, \varepsilon} $$
  and $\tilde{h}_j^{C, \varepsilon}, \hat{h}_j^{C, \varepsilon} $ are monotone sequences (increasing and decreasing, respectively) converging to $h^{C, \varepsilon}$.\\
 Moreover, since  $h_j^{C, \varepsilon}=0$ if  $\psi_j\leq -C$ and $h^{C, \varepsilon}=0$ if  $\varphi\leq -C$, by the locality of  the Monge-Amp\`ere measure with respect to the plurifine topology \cite[Section 1.2]{BEGZ10}, we have 
 $$ h_j^{C, \varepsilon} g  \, \theta_{\psi_j}^n  =  h_j^{C, \varepsilon} g \, {\theta}_{\max( \psi_j,-C)}^n, \quad {\rm and} \quad  h^{C, \varepsilon} g  \, \theta_\varphi^n  =  h^{C, \varepsilon} g \, {\theta}_{\max( \varphi,-C)}^n. $$
Combining all the above we get

 \begin{eqnarray*}
  \int_X g\, \theta_\varphi^n  & \geq  & \int_{ X}h^{C, \varepsilon}   g\, \theta_\varphi^n \\
  &=&  \int_{ X} h^{C, \varepsilon} g \, {\theta}_{\max( \varphi,-C)}^n\\ 
   & = & \lim_{j \rightarrow + \infty}  \int_{X }  h_j^{C, \varepsilon} g\, {\theta}_{\max( \psi_j,-C)}^n  \\
& =& \lim_{j \rightarrow + \infty}  \int_{ X } h_j^{C, \varepsilon}  g \,\theta_{\psi_j}^n  \\
&\geq &  \lim_{j \rightarrow + \infty}  \int_{ \{ \varphi = f \} } h_j^{C, \varepsilon} g \, \theta_f^n \\
&= &   \int_{ \{ \varphi = f \} } h^{C, \varepsilon}  g\,\theta_f^n
 \end{eqnarray*}
 where the last inequality follows from \eqref{ineq1} while the last identity follows from the fact that convergence in capacity implies $L^1$ -convergence \cite[Lemma 4.24]{GZ17}. Observe that since we choose $\min f  >  -C$, the functions $ h^{C, \varepsilon} >0$ on   $\{ \varphi = f \} $. Letting $\varepsilon\rightarrow 0$, the dominated convergence theorem gives that
 $$  \int_X g\, \theta_\varphi^n   \geq  \int_{ \{ \varphi = f \} } g\,\theta_f^n.$$
Since the above inequality holds for any non negative continuous function $g$, by Riesz' representation theorem \cite[Theorem 7.2.8]{Cohn} we then we derive the inequality between measures
\begin{equation}  \label{ineqma}
\theta_\varphi^n    \geq {\bf 1}_{ \{ \varphi = f \}} \theta_f^n.
 \end{equation}
 Then, by \cite[Lemma 4.5]{DDL4} we get the equality

\begin{equation}\label{eqma1}
{\bf 1}_{ \{ \varphi = f \}}  \theta_\varphi^n   = {\bf 1}_{ \{ \varphi = f \}}\theta_f^n .  
 \end{equation}
%

\noindent {\bf Step 2.} The next step is to prove the equality in \eqref{eqma1} when $\theta$ is merely pseudoeffective and not necessarily K\"ahler. Also, we assume that there exists $A > 0$ such that $\theta + A \omega$ is a K\"ahler form and  $f$ is   $(\theta +A \omega)$-psh. Observe that, since $\varphi $ is $\theta$-psh function, then $\varphi$ is also $\theta + t \omega$-psh, for $t \geq 0.$
Let $g\in C^0(X, \R)$ and consider the function 
 \[ Q(t) := \int_{ \{ \varphi = f \} } g ( \theta+ t \omega + dd^c \varphi )^n - \int_{ \{ \varphi = f \} } g ( \theta + t \omega + dd^c f )^n \] 
 defined for $t \geq 0.$ Then by multilinearity of the non-pluripolar product and the multilinearity of the product of forms, it is clear that  $Q(t)$ is a polynomial in $t$ of the form:
\[ Q(t) = \sum_{j=0}^{n}  \left(   {n\choose j} \int_{ \{ \varphi = f \} } g \left(   ( \theta  + dd^c \varphi )^j  - ( \theta  + dd^c f )^j \right)  \wedge \omega^{n-j} \right)  t^{n-j}. \] 
Thanks to (\ref{eqma}) we can infer that for any $t>A$
$$ {\bf 1}_{ \{ \varphi = f \}}  (\theta+t\omega+dd^c \varphi)^n   = {\bf 1}_{ \{ \varphi = f \}}(\theta+t\omega+dd^c f)^n.   $$
This implies that the polynomial  $Q(t)$ is identically zero for $t > A,$ hence $Q(t) \equiv 0$. It the follows $Q(0) = 0.$ Since $g\in C^0(X, \R)$ is arbitrary we have the desired equality between measures.\\

\noindent {\bf Step 3.} We now prove equality (\ref{eqma1}) when $f\in C^{1, \overline{1}}$ and not necessarily qpsh. Choose $A > 0$ such that  
$\theta + A \omega$ is a K\"ahler form. Since  the function $\varphi$ is 
$(\theta + A\omega)$-psh, then 
\begin{equation}\label{ineq2} 
\varphi \leq P_{\theta + A \omega}(f) \leq f. \end{equation} 
In particular 
\begin{equation}\label{inclus2} 
\{  \varphi = f \} = \{ \varphi = P_{\theta + A \omega}(f)  \} \cap \{ P_{\theta + A \omega}(f) =  f \}.
\end{equation} 
Observe that, even if $P_{\theta + A \omega}(f)$ and $f$ are not $\theta$-psh  they are both $C^{1, \overline{1}}$ \cite[Theorem 2.5]{DR}. Lemma \ref{min} applied to  the functions $f_1 =  f_2 = f $ ensures that the functions $P_{\theta + A \omega}(f)$ and $f$ are equal up to second order at almost every point on the set  ${ \{ P_{\theta + A \omega} (f)  = f \}} $. Hence
\begin{equation}\label{eqP} 
{\bf 1}_{ \{ P_{\theta + A \omega} (f)  = f \}}  \theta_{P_{\theta + A\omega}(f)}^n   =     
 {\bf 1}_{ \{ P_{\theta + A \omega} (f)  = f \}}  { \theta_f }^n.      \end{equation}
Moreover, since $P_{\theta + A \omega}(f)$ is $(\theta + A \omega
)$-psh, Step 2 ensures that 

\begin{equation}\label{eqP2} 
{\bf 1}_{ \{ \varphi = P_{\theta + A \omega} (f) \}}  \theta_\varphi^n   =  {\bf 1}_{ \{  \varphi = P_{\theta + A \omega} (f)   \}}  \theta_{P_{\theta + A\omega}(f)}^n. \end{equation}
If we multiply (\ref{eqP2}) by $ {\bf 1}_{ \{  P_{\theta + A \omega} (f) = f  \}}$ and use (\ref{eqP}) and (\ref{inclus2}) we derive (\ref{eqma1}).
\\
\noindent {\bf Step 4.} We now  prove the general case.
Let $ \lambda  = (\lambda_1, \lambda_2, \ldots, \lambda_n) \in (\R^+)^n$. We set
$\theta^{ \lambda } :=\sum_j  \lambda_j \theta^j,$   $ \varphi_{ \lambda } := \sum_j \lambda_j \varphi_j$, $ f_{ \lambda } := \sum_j \lambda_j f_j.$ Observe that  $\varphi_\lambda$ is $\theta^\lambda$-psh and $f_\lambda$ is still $C^{1,\bar{1}}$ and that
\begin{equation}\label{inclusion1}
\bigcap_{j=1}^n \{ \varphi_j = f_j \} \subseteq \{  \varphi_\lambda=f_\lambda  \}.
 \end{equation}
By Step 3 and \eqref{inclusion1} we get
$${\bf 1}_{\bigcap_j\{ \varphi_j = f_j \} }(\theta^\lambda_{\varphi_\lambda})^n = {\bf 1}_{\bigcap_j\{ \varphi_j = f_j \} }(\theta^\lambda_{f_\lambda})^n. $$
Using the multilinearity of the non-pluripolar product we get an identity between two polinomials in the variables $\lambda_1, \ldots, \lambda_n$. Comparing the coefficients of $\lambda_1 \cdots \lambda_n$ we obtain
$$ {\bf 1}_{\bigcap_j\{ \varphi_j = f_j \} } \theta^1_{\varphi_1} \wedge \ldots \wedge \theta^n_{\varphi_n} = {\bf 1}_{\bigcap_j\{ \varphi_j = f_j \} } \theta^1_{f_1} \wedge \ldots \wedge \theta^n_{f_n}.$$
\end{proof}

\begin{remark}\label{example}
One can not expect Theorem \ref{main} to hold when the barrier function $f$ is singular. The following counterexample shows indeed that \eqref{eqma} does not hold when $f$ is merely continuous.\\
Let $\mathbb{B}\subset X$ be a small open
ball and let $\rho$ be a smooth potential such that $\omega=dd^c \rho$ in a neighbourhood of $\overline{\mathbb{B}}$. We solve the Dirichlet problem $$(dd^c(\rho+v))^n=0\quad {\rm in }\; \mathbb{B}, \qquad v|_{\partial \mathbb{B}}=0.  $$
Since the boundary data is continuous, \cite[Proposition 1.6 and Corollary 1.17]{GZ17} garantees the existence of a continuous solution $v\geq 0$ which is $\omega$-psh in $\mathbb{B}$. We then define 
$$f:=
\begin{cases}
v \quad {\rm in }\; \mathbb{B}\\
0 \quad {\rm in }\; X\setminus \mathbb{B}.
\end{cases}
$$
By construction $f$ is a continuous $\omega$-psh function and $f\geq 0$. On the other hand we observe that
$$\int_{X\setminus \mathbb{B}} \omega_f^n = \int_{X} \omega_f^n = \int_{X} \omega^n > \int_{X\setminus \mathbb{B}} \omega^n.$$
Since $\{f=0\}\subseteq X\setminus \mathbb{B}$, we then deduce that the two measures ${\bf 1}_{\{f=0\}} \omega^n$ and ${\bf 1}_{\{f=0\}} \omega_f^n$ can not coincide.
\end{remark}
\begin{coro} \label{envelopes} 
Let $\varphi\in \psh(X, \theta)$ and $f\in C^{1, \overline{1}}(X)$ be such that $\varphi\leq f$. We have:
\begin{itemize}
\item[i)] $ \theta_{ P_{\theta}(f)}^n = { \bf 1}_{ \{ P_{\theta}(f)  = f \}} \theta _f^n $.
 \item[ii)]  ${\theta}_{P[\varphi](f)}^n=  { \bf 1}_{ \{ P[\varphi](f) = f \}} \theta_f^n$.
  \item [iii)]
$ {\bf{1}}_{\{ \varphi = P_{\theta}(f) \}} \ \theta_{ P_{\theta}(f)}^n =   {\bf{1}}_{\{ \varphi = P_{\theta}(f) \}} \,\theta_\varphi^n \;$  and   $  \;{\bf{1}}_{\{ \varphi =  P_{\theta}[\varphi](f)  \}}\, \theta_{P_{\theta}[\varphi](f) }^n =   {\bf{1}}_{\{ \varphi =  P_{\theta}[\varphi](f)  \}}\, \theta_{\varphi}^n. $ Moreover, the measure $\theta_\varphi^n$ is supported on the set $\{\varphi=f\}\cup \{\varphi <P_\theta[\varphi] (f)\}$.

\item[iv)]  The following conditions are equivalent: 
    \begin{enumerate}
  \item[1)] $ {\bf 1}_{ \{ \varphi < P_\theta[\varphi](f) \} }   {\theta}_{\varphi}^n = 0 $
  \item[2)]  either  $ \theta_\varphi^n = 0,$ or $ \varphi = P_\theta[\varphi](f)$.
     \end{enumerate} 
\item[v)] Assume  $ \theta_\varphi^n > 0 $. The set  $\{ P_\theta[\varphi](f) = f,\, \varphi< f\}$ has  measure zero w.r.t. $\theta_f^n$ if and only if $\varphi = P[\varphi](f).$ 
\end{itemize}
 
 \end{coro}

\begin{proof}
Then statement in $(i)$ immediately  follows from Theorem \ref{main} and \cite[Proposition 2.16]{DDL1}. \\
Since $P_\theta[\varphi](f) = P_{\theta}[\varphi](P_\theta(f))$ and $\varphi\leq P_\theta(f)$, by \cite[Theorem 3.8]{DDL2} we have
$$\theta^n_{P_\theta[\varphi](f)} \leq  { \bf 1}_{\{  P_\theta[\varphi](f) = P_{\theta}(f)\}} \theta_{P_{\theta}(f)}^n = { \bf 1}_{ \{ P_\theta[\varphi](f) = f \}} \theta _f^n,$$
where in the last equality we used $(i)$. The opposite inequality follows from Theorem \ref{main}. This proves $(ii)$. \\Let's now prove $(iii)$. By $(i)$ and by Theorem \ref{main}, 
\begin{equation}\label{ineq3}
   {\bf{1}}_{\{ \varphi =  P_{\theta}(f)  \}}  \theta_{P_{\theta}(f)}^n =   {\bf{1}}_{\{  \varphi = P_{\theta}(f) = f \}} \theta_f^n =  {\bf{1}}_{\{ \varphi  = P_{\theta}(f) = f \}} \theta_\varphi^n \leq  {\bf{1}}_{\{ \varphi =  P_{\theta}(f)\}} \theta_\varphi^n.  
  \end{equation} 
  The other inequality is given by \cite[Lemma 4.5]{DDL4}. In particular the inequality in \eqref{ineq3} is in fact an equality, hence
  $$ {\bf{1}}_{\{   P_{\theta}(f)=\varphi <f  \}}  \theta_\varphi^n =0 .$$
Using $(ii)$ and Theorem \ref{main}, the same arguments of above give $$   {\bf{1}}_{\{ \varphi =  P_{\theta}[\varphi](f)  \}} \theta_{P_{\theta}[\varphi](f) }^n =   {\bf{1}}_{\{ \varphi =  P_{\theta}[\varphi](f)  \}} \theta_{\varphi}^n\qquad {\rm and }\qquad  {\bf{1}}_{\{   P_{\theta}[\varphi](f)=\varphi <f  \}}  \theta_\varphi^n =0.$$
 We now prove $(iv)$. If $\theta_{\varphi}^n=0$ or $\varphi = P_\theta[\varphi](f)$ then clearly $ {\bf 1}_{ \{ \varphi < P_\theta[\varphi](f) \} }   {\theta}_{\varphi}^n = 0 $. This proves that 2) implies 1). Viceversa we assume 1) and that $\int_X \theta_\varphi^n > 0.$ By \cite[Remark 2.5]{DDL2}
\begin{equation}\label{ineq4}
\int_X \theta_\varphi^n  = \int_X \theta_{P_\theta[\varphi](f)}^n=\int_X \theta_{P_\theta[\varphi]}^n.
\end{equation}
 The domination principle \cite[Proposition 3.11]{DDL2} gives the conclusion.\\
We finally prove $(v).$  By assumption and by $(ii)$ we have
$$\int_{\{  \varphi < P_\theta[\varphi](f)  \}} \theta_{P_\theta[\varphi](f)}^n=\int_{\{ \varphi < P_\theta[\varphi](f)\} \cap \{  P_\theta[\varphi](f) = f\}} \theta_f^n=0.$$
Using \eqref{ineq4}, $(iii)$ and the above we get that
$$ \int_{\{  \varphi < P_\theta[\varphi](f)  \}} \theta_{\varphi}^n=\int_{X} \theta_{\varphi}^n- \int_{\{  \varphi = P_\theta[\varphi](f)  \}} \theta_{\varphi}^n =\int_{\{  \varphi < P_\theta[\varphi](f)  \}} \theta_{P_\theta[\varphi](f)}^n=0.$$  Once again the conclusion follows from the domination principle \cite[Proposition 3.11]{DDL2}.
\end{proof}

Next we give a formula relating the Monge-Amp\`ere measure of $P_{\theta}(f_1, \ldots,f_k)$ to the  $(n,n)$-forms  $\theta_{f_j}^n, \ 1 \leq j \leq k.$  
Set $R := \{  P_{\theta}(f_1, \ldots,f_k)= \min \{ f_1, \ldots, f_k \}) \},$ 
and let $\mathcal{I}$ be the family of all non empty subsets of $ \{1, \ldots, k \}.$ For 
$I \in \mathcal{I},$ we let 
$$ R_I =\{x\in X \,:\, f_{i} (x)= \min(f_1, \ldots, f_k)(x)= P_{\theta}(f_1, \ldots,f_k)(x) \; {\rm iff} \; i\in I\} . 
$$
Then $\{ R_I \}_{I \in \mathcal{I} }$ gives a partition of $R$ by Borel sets.

\begin{prop}  \label{mu_s}  
Let $f_1, \ldots f_k\in C^{1, \overline{1}}(X)$. For $I \in \mathcal{I}, I = \{ i_1, i_2, \ldots,i_r \} $  we have:
\[ {\bf 1}_{R_I} \theta_{ f_{i_1}}^n =  \ldots = {\bf 1}_{R_I}  \theta_ { f_{i_r} }^n = {\bf 1}_{R_I} \theta_{P_{\theta}(f_1, \ldots,f_k)}^n:=\mu_I . \] 
Moreover,    
 $ \theta_{P_{\theta}(f_1, \ldots,f_k)}^n= \sum_ {I \in \mathcal{I}}  \mu_I. $
\end{prop}

\begin{proof}

By \cite[Proposition 2.16]{DDL1} the measure $\theta_{P_{\theta}(f_1, \ldots,f_k)}^n$ is supported on the contact set $R.$ Moreover, for any $h=1, \dots, r$, we have $P_{\theta}(f_1, \ldots,f_k))  \leq f_{i_h}$ with equality on the set $R_I.$ The conclusion follows from Theorem \ref{main}. 
\end{proof}

\noindent{\sc Sorbonne Universit\'e}\\
{\tt eleonora.dinezza@imj-prg.fr}\vspace{0.1in}\\
\noindent {\sc Universit\'a di Roma TorVergata}\\
{\tt trapani@axp.mat.uniroma2.it}
\end{document}